\documentclass[12pt,draft,reqno]{amsart}
\usepackage{amsmath}
\usepackage{amssymb}
\usepackage{amsthm}
\usepackage{bm}
\numberwithin{equation}{section}
\newtheorem{theorem}{Theorem}[section]
\newtheorem{lemma}[theorem]{Lemma}
\newtheorem{proposition}[theorem]{Proposition}
\newtheorem{corollary}[theorem]{Corollary}

\newtheorem{thm}{Theorem}

\theoremstyle{definition}
\newtheorem*{remark}{Remark}
\newtheorem*{conj}{Conjecture}

\allowdisplaybreaks
\begin{document}
\title[Zeros of the quadrilateral zeta function on the critical line]{Functional equation and zeros on the critical line of the quadrilateral zeta function}
\author[T.~Nakamura]{Takashi Nakamura}
\address[T.~Nakamura]{Department of Liberal Arts, Faculty of Science and Technology, Tokyo University of Science, 2641 Yamazaki, Noda-shi, Chiba-ken, 278-8510, Japan}
\email{nakamuratakashi@rs.tus.ac.jp}
\urladdr{https://sites.google.com/site/takashinakamurazeta/}
\subjclass[2010]{Primary 11M35, Secondary 11M26}
\keywords{Functional equation, converse theorems, Quadrilateral zeta function, Zeros on the critical line}
\maketitle

\begin{abstract}
For $0 < a \le 1/2$, we define the quadrilateral zeta function $Q(s,a)$ using the Hurwitz and periodic zeta functions and show that $Q(s,a)$ satisfies Riemann's functional equation studied by Hamburger, Heck and Knopp. Moreover, we prove that for any $0 < a \le 1/2$, there exist positive constants $A(a)$ and $T_0(a)$ such that the number of zeros of the quadrilateral zeta function $Q(s,a)$ on the line segment from $1/2$ to $1/2 +iT$ is greater than $A(a) T$ whenever $T \ge T_0(a)$. 
\end{abstract}

 \tableofcontents
 
\section{Introduction and Statement of the Main Results}
\subsection{Main results}
For $0< a \le 1$, define the Hurwitz zeta function $\zeta (s,a)$ by 
\[
\zeta (s,a) := \sum_{n=0}^\infty \frac{1}{(n+a)^s}, \qquad \sigma >1,
\]
and the periodic zeta function ${\rm{Li}} (s,a)$ by 
\[
{\rm{Li}} (s,a) := \sum_{n=1}^\infty \frac{e^{2\pi ina}}{n^s}, \qquad \sigma >1.
\]
The Dirichlet series of $\zeta (s,a)$ and ${\rm{Li}} (s,a)$ converge absolutely in the half-plane $\sigma >1$ and uniformly in each compact subset of this half-plane. Moreover, the Hurwitz zeta function has analytic continuation to ${\mathbb{C}}$ except $s=1$, where there is a simple pole with residue $1$ (e.g.,~\cite[Chapter 12]{Apo}). In contrast, the Dirichlet series of the function ${\rm{Li}} (s,a)$ with $0<a<1$ converges uniformly in each compact subset of the half-plane $\sigma >0$ (e.g.,~\cite[p.~20]{LauGa}). Furthermore, the function ${\rm{Li}} (s,a)$ with $0<a<1$ is analytically continuable to the whole complex plane (e.g., \cite[Chapter 2.2]{LauGa}). We clearly have $\zeta (s,1) = {\rm{Li}} (s,1) = \zeta (s)$, where $\zeta (s)$ is the Riemann zeta function. 

For $ 0 <a \le 1/2$, we define the quadrilateral zeta function $Q(s,a)$ as
\begin{equation}\label{eq:defQSA}
2Q(s,a) := \zeta (s,a) + \zeta (s,1-a) +{\rm{Li}} (s,a) + {\rm{Li}} (s,1-a).
\end{equation}
Based on the facts mentioned above, the function $Q(s,a)$ can be continued analytically to the whole complex plane except $s=1$. The first main theorem is the following functional equation of the quadrilateral zeta function $Q(s,a)$.
\begin{theorem}\label{th:FE}
For $ 0 <a \le 1/2$, it holds that
\begin{equation}\label{eq:zfe1q}
Q(1-s,a) =\frac{2\Gamma (s)}{(2\pi )^s} \cos \Bigl( \frac{\pi s}{2} \Bigr) Q(s,a).
\end{equation}
\end{theorem}
Moreover, we show the following, which implies that $Q(s,a)$ has infinitely many zeros on the critical line $\sigma =1/2$. 
\begin{theorem}\label{th:main}
For any $0 < a \le 1/2$, there exist positive constants $A(a)$ and $T_0(a)$ such that the number of zeros of $Q(s,a)$ on the line segment from $1/2$ to $1/2 +iT$ is greater than $A(a) T$ whenever $T \ge T_0(a)$. 
\end{theorem}

We share some remarks on the functional equation and zeros on the critical line of zeta functions in the next three subsections. Note that the quadrilateral zeta function $Q(s,a)$ also has the following remarkable properties. From \cite[(2.4)]{NRCQZ}, it holds that
\[
Q(0,a) = -1/2 = \zeta (0) \quad \mbox{for all} \quad 0 < a \le 1/2. 
\]
For $n \in {\mathbb{N}}$, it is shown in \cite[Corollary 3.7]{NaIn} that $Q(-n,a)$ and $\pi^{-2n}Q(2n,a)$ are rational functions of $e^{2\pi ia}$ with rational coefficients. By (\ref{eq:zfe1q}), the function $Q(s,a)$ has simple zeros at the negative even integers.  Furthermore, it is proved in \cite[Theorem 1.1]{NRCQZ} that there exists $a_0 = 0.1183751396...$ such that \\
{\rm{(1)}} $Q(\sigma, a_0)$ has a unique double real zero at $\sigma = 1/2$ when $\sigma \in (0,1)$,\\ 
{\rm{(2)}} for any $a \in (a_0,1/2]$, the function $Q(\sigma, a)$ has no real zero in $\sigma \in (0,1)$,\\
{\rm{(3)}} for any $a \in (0,a_0)$, $Q(\sigma, a)$ has at least two real zeros in $\sigma \in (0,1)$. \\
In addition, for $0 <a \le 1/2$, it is shown in \cite[Proposition 1.5]{NRCQZ} that
\[
N\bigl( T, Q(s,a)\bigr) = \frac{T}{\pi} \log T - \frac{T}{\pi} \log (2 e \pi a^2) + O_a(\log T),
\]
where $N(T, F)$ is the number of non-real zeros of a function $F(s)$ with $|\Im (s)| < T$ when $T$ is sufficiently large. Moreover, we prove in  \cite[Proposition 1.4]{NRCQZ} that $Q(s,a)$ has infinitely many complex zeros in the region of absolute convergence and the critical strip when $a \in {\mathbb{Q}} \cap (0,1/2) \setminus \{1/6, 1/4, 1/3\}$. 

\subsection{Zeros of zeta functions on the critical line}
The famous Riemann hypothesis asserts that the real part of every non-real zero of the Riemann zeta function is $1/2$. The study to establish the lower bound for the number of zeros of $\zeta (s)$ on the critical line $\sigma = 1/2$ has long history. Denote by $N_{\rm{Ri}} (T)$ the number of zeros $\rho = \beta + i\gamma$ of the Riemann zeta function $\zeta (s)$ with $\beta =1/2$ and $0 < \gamma \le T$. In 1914, Hardy proved that $N_{\rm{Ri}} (T) \to \infty$ if $T \to \infty$. Later, Hardy and Littlewood \cite{HL} showed the following (see also \cite[Chapter 11.2]{ERZ} and \cite[Chapter 10.7]{Tit}):
\begin{thm}[{Hardy and Littlewood \cite[Theorem A]{HL}}]\label{th:HL}
There are constants $A>0$ and $T_0 >0$ such that $N_{\rm{Ri}} (T) \ge A T$ whenever $T > T_0$.
\end{thm}

In 1942, Selberg proved that there exists $A>0$ such that 
\[
N_{\rm{Ri}} (T) \ge A T\log T. 
\]
Note that the numerical value of the constant $A$ in Selberg's theorem was very small. However, Levinson \cite{Lev} greatly improved Selberg's result and showed that $A \ge 1/3$. Furthermore, Conrey \cite{Con} proved that $A \ge 0.4088$. The current (June 2021) best result, which was proved by K\"uhn, Robles, and Zeindler \cite{KRZ}, for the lower bound of $A$ is
\[
A \ge 0.410725 .
\]

It is well-known that the Riemann zeta function $\zeta (s)$ does not vanish in the region of absolute convergence by the Euler product. Next, we review some facts about the zeros on the vertical line $\sigma =1/2$ of the Epstein and Hurwitz zeta functions, which have complex zeros in the half-plane $\sigma >1/2$ (e.g., \cite[Chapter 7.4.3]{KV} and \cite[Chapter 8.4]{LauGa}).  

Let $B(x,y)= a x^2 + bxy + c y^2$ be a positive definite integral binary quadratic form, and denote by $r_{\!B}(n)$ the number of solutions of the equation $B(x,y) = n$ in integers $x$ and $y$. Then, the Epstein zeta function for the form $B$ is defined by the ordinary Dirichlet series
\[
\zeta_B(s) := \sum_{(x,y) \in {\mathbb{Z}}^2\setminus (0,0)} \frac{1}{B(x,y)^s}
= \sum_{n=1}^\infty \frac{r_{\!B}(n)}{n^s}
\]
for $\sigma >1$. It is widely known that the function $\zeta_B(s)$ admits analytic continuation into the entire complex plane except for a simple pole at $s = 1$ with residue $2\pi \Delta^{-1}$, where $\Delta := \sqrt{4ac-b^2}$ (e.g., \cite[Section 1]{Ep}). Moreover, the function $\zeta_B(s)$ satisfies the functional equation
\[
\Bigl( \frac{\Delta}{2\pi} \Bigr)^s \Gamma (s) \zeta_B(s) = 
\Bigl( \frac{\Delta}{2\pi} \Bigr)^{1-s} \Gamma (1-s) \zeta_B(1-s) .
\]

Denote by $N_{\rm{Ep}} (T)$ the number of zeros of the Epstein zeta function $\zeta_B (s)$ on the critical line and whose imaginary part is smaller than $T>0$. In 1935, Potter and Titchmarsh \cite{PT} showed that $N_{\rm{Ep}} (T) \gg T^{1/2-\varepsilon}$. Subsequently, Sankaranarayanan \cite{Sa} obtained $N_{\rm{Ep}} (T) \gg T^{1/2}/\log T$, and Jutila and Srinivas \cite{JS} proved that $N_{\rm{Ep}} (T) \gg T^{5/11-\varepsilon}$. As the current (June, 2021) best result, Baier, Srinivas, and Sangale \cite{BSS} showed that
\[
N_{\rm{Ep}} (T) \gg T^{4/7-\varepsilon}.
\]

A key to the proof of the estimation $N_{\rm{Ep}} (T) \gg T^{4/7-\varepsilon}$ shown in \cite{BSS} is the first power mean of an ordinary Dirichlet series $\sum_{n=1}^\infty b_n n^{-s}$ with $b_n \in {\mathbb{C}}$ satisfying certain conditions cannot be too small. By zeros on the critical line and the functional equation of the Epstein zeta function $\zeta_B (s)$ mentioned above, the quadrilateral zeta function $Q(s,a)$ has many analytical properties in common with the Epstein zeta function (and the Riemann zeta function). It should be mentioned that the gamma factor of $Q(s,a)$ does not depend on the parameter $0 < a < 1/2$ from Theorem \ref{th:FE}, but the gamma factor of $\zeta_B (s)$ depends on the discriminant $\Delta$ of the positive definite integral binary quadratic form $B(x,y)$. Furthermore, we can see that the lower bound for the zeros of $Q(s,a)$ on the critical line is better than that of $\zeta_B (s)$ by virtue of Theorem \ref{th:main} at present. 

For the Hurwitz zeta function $\zeta (s,a)$ with $a=1/3, 2/3$, $1/4, 3/4, 1/6,$ or $5/6$, Gonek \cite{Gon} showed that there exists a constant $0<c<1$ such that the number of zeros (including multiplicities) of $\zeta (s,a)$ on the segment $(1/2, 1/2+iT)$ is $\le (c+o(1))(T/2\pi) \log T$ as $T$ tends to infinity. Moreover, he concluded with the following conjecture:
\begin{conj}[Gonek \cite{Gon}]
If $0<a<1$ is rational and $a \ne 1/2$, then the Hurwitz zeta function $\zeta (s,a)$ has $\ll T$ zeros on the segment $(1/2, 1/2+iT)$.
\end{conj}
Based on this conjecture and the facts above, we can guess that proving the existence of $\gg T$ zeros on the line segment $(1/2, 1/2+iT)$ of the Hurwitz or Epstein zeta functions is difficult because these zeta functions have no Euler product in general. However, we show that the quadrilateral zeta function $Q(s,a)$ has $\gg T$ zeros on the line segment $(1/2, 1/2+iT)$ even though $Q(s,a)$ cannot be written as an Euler product (see (\ref{eq:defQSA})). 

\begin{remark}
The quadrilateral zeta function $Q(s,a)$ with $a \in {\mathbb{Q}}$ can be essentially expressed as a linear combination of Euler products from (\ref{eq:qq}). Hence, under the GRH and some assumptions on well-spacing of zeros for Dirichlet $L$-functions, we could show that $Q(s,a)$ with $a \in {\mathbb{Q}} \cap(0,1/2) \setminus \{1/3,1/4,1/6\}$ have 100 \% of zeros on the line $\sigma = 1/2$ if we could replace the function $\sum_{j=1}^N b_j L(s,\chi_j)$, where $b_j \in {\mathbb{R}} \setminus \{ 0\}$, in Bombieri and Hejhal \cite[Theorem A]{BH} by the function $\sum_{j=1}^N (\beta_{1j} + \beta_{2j} q^s ) L(s,\chi_j)$, where $\beta_{1j}, \beta_{2j} \in {\mathbb{C}} \setminus \{ 0\}$ and $q$ is a natural number. However, it seems to be extremely difficult for us to relax their assumptions in \cite[Theorem A]{BH} as above (even when $\beta_{2j}=0$). It is worth noting that $Q(s,a)$ with $a \in {\mathbb{R}} \setminus {\mathbb{Q}}$ can be expressed as neither an ordinary Dirichlet series nor a linear combination of Euler products. Despite these facts, we can prove Theorem \ref{th:main} by modifying the proof of Hardy and Littlewood's classical ideas in \cite[Sections 2, 3, and 4]{HL} (see also \cite[Chapter 11.2]{ERZ} and \cite[Chapter 10.7]{Tit}). 
\end{remark}

\subsection{Hamburger's, Hecke's and Knopp's Theorems}
It is widely known that $\zeta (s)$ satisfies the functional equation 
\begin{equation}\label{eq:RZfe}
\zeta(1-s) =\frac{2\Gamma (s)}{(2\pi )^s} \cos \Bigl( \frac{\pi s}{2} \Bigr) \zeta(s).
\end{equation}
As the first converse theorem for the Riemann zeta function $\zeta (s)$, Hamburger \cite{Ham} proved that $\zeta(s)$ is characterized by the functional equation (\ref{eq:RZfe}) (see also Siegel \cite{Sie} and Titchmarsh \cite[Chapter 2.13]{Tit}). 
\begin{thm}[Hamburger {\cite[Satz 1]{Ham}}]
Let $G(s)$ be an entire function of finite order, $P(s)$ be a polynomial, and suppose that
\begin{equation}
f(s) := \frac{G(s)}{P(s)} = \sum_{n=1}^\infty \frac{a(n)}{n^s}, \qquad a(n) \in {\mathbb{C}}, \tag{H1}
\end{equation}
the series being absolutely convergent for $\sigma >1$. Assume that
\begin{equation}
\pi^{-s/2} \Gamma \Bigl( \frac{s}{2} \Bigr) f(s) = \pi^{-(1-s)/2} \Gamma \Bigl( \frac{1-s}{2} \Bigr) f(1-s).
\tag{H2}
\end{equation}
Then, we have $f(s) = C\zeta(s)$, where $C$ is a constant.
\end{thm}

Hecke \cite[Section 1]{He} showed that Hamburger's Theorem can be rewritten as: the following three conditions characterize $\zeta (2s)$ up to a constant factor.
\begin{flushleft}
{\bf{${}$ -(1)-}} The function $\phi(s)$ is meromorphic, and $P(s)\phi(s)$ is an entire function of finite genus with a suitable polynomial $P(s)$. \\
{\bf{${}$ -(2)-}} The function $R(s)= \pi^{-s} \Gamma(s) \phi(s)$ satisfies the functional equation\\ $R(s)=R(1/2 -s)$. \\
{\bf{-(3a)-}} Both functions $\phi(s)$ and $\phi(s/2)$ can be expanded in a Dirichlet series that converges in a half-plane.\\
\end{flushleft}

Moreover, Hecke \cite{He} proved that the expressibility of $\phi(s/2)$ as a Dirichlet series in -(3a)- can be replaced by the following restriction on the poles of $\phi(s)$. More precisely, he showed that $\zeta (2s)$ (up to a constant factor) is uniquely determined by -(1)-, -(2)-, and
\begin{flushleft}
{\bf{-(3b)-}} The function $\phi(s)$ can be expanded in a Dirichlet series that converges somewhere and the only pole allowed for $\phi(s/2)$ is $s=1$. 
\end{flushleft}

It is quite natural to relax the conditions introduced by Hecke. Knopp \cite{Kn} showed the following, which implies that there are infinitely many linearly independent solutions if we drop the pole condition -(3b)- above by using the Riemann-Hecke correspondence between ordinary Dirichlet series with functional equations and modular forms or the generalized Poincar\'e series.
\begin{thm}[Knopp {\cite[Theorem 1]{Kn}}]
Let $\sigma_0 \ge 1/4$ and ${\mathcal{A}} (\sigma_0)$ be the space of all rational functions $A(s)$ with poles restricted to the vertical strip $1/2 -\sigma_0 \le \Re (s) \le \sigma_0$ and satisfying $A(1/2-s)=A(s)$. Let $A_H(\sigma_0)$ be the subspace of $A$ in $A(\sigma_0)$ such that $R(s) - A(s)$ is entire for some $R(s) = \pi ^{-s} \Gamma(s) \phi (s)$ satisfying -(1)-, -(2)-, and 
\begin{flushleft}
{\bf{-(3)-}} The function $\phi(s)$ can be expanded in a Dirichlet series that converges somewhere.
\end{flushleft}
Then with ${\mathbb{N}} \ni n > [\sigma_0/2+3/8]+2$ and $A_1, \ldots , A_n \in A(\sigma_0)$, some nontrivial linear combination of the $A_j$ is in $A_H(\sigma_0)$. 
\end{thm}

According to the theorems by Hamburger, Hecke, and Knopp, we can see that the conditions to characterize $\zeta (s)$ introduced by Hamburger or Hecke are so polished that a slight weakening of their conditions leads to infinitely many counterexamples, as mentioned by Knopp. Note that Knopp's theorem does not provide any explicit representation for the coefficients of $a(n)$ of the Dirichlet series satisfying condition -(3)-. However, as analogues or improvements to Knopp's Theorem, in the next subsection, we show that the zeta function $Q(s,a)$ defined explicitly in Section 1.1 fulfills the assumption -(2)- and some modified conditions of -(1)- and -(3a)- or -(3b)-.  

\subsection{Variations of Knopp's Theorem}
Now, we consider some variations of Knopp's Theorems, namely, we properly modify conditions -(1)-, -(3a)-, and -(3b)- introduced by Hecke and prove that $Q(s,a)$ fulfills the reshaped conditions. 

First, we have the following immediately from Theorem \ref{th:FE}. 

\begin{corollary}\label{cor:1}
The function $Q(2s,a)$ satisfies -(1)-, -(2)- and
\begin{flushleft}
{\bf{-(3a')-}} Both functions $\phi(s)$ and $\phi(s/2)$ can be expanded in a {\sc{general}} Dirichlet series that converges in a half-plane.
\end{flushleft}
\end{corollary}

Next, let $\varphi$ be the Euler totient function and $\chi$ be a primitive Dirichlet character of the conductor of $q$. Let $L(s,\chi) := \sum_{n=1}^\infty \chi(n) n^{-s}$ be the Dirichlet $L$-function. Then, for $0 < r < q$, where $q$ and $r$ are relatively prime integers, we have
\begin{equation}\label{eq:k1}
\zeta (s,r/q) = \sum_{n=0}^\infty \frac{1}{(n \!+\! r/q)^s} = \sum_{n=0}^\infty \frac{q^s}{(r \!+\! qn)^s} =
\frac{q^s}{\varphi (q)} \! \sum_{\chi \!\!\!\! \mod q} \! \overline{\chi} (r) L(s,\chi).
\end{equation}
In addition, let $G(r, \overline{\chi})$ denote the (generalized) Gauss sum $G(r,\overline{\chi}) := \sum_{n=1}^q$ $\overline{\chi}(n)e^{2\pi irn/q}$ associated with a Dirichlet character $\overline{\chi}$. Then we have 
\begin{equation}\label{eq:k2}
{\rm{Li}} (s, r/q) = q^{-s} \sum_{n=1}^q e^{2\pi irn/q} \zeta (s,n/q) =
\frac{1}{\varphi (q)} \sum_{\chi \!\!\! \mod q} G(r,\overline{\chi}) L(s,\chi).
\end{equation}
Hence, from (\ref{eq:k1}) and (\ref{eq:k2}), it holds that
\begin{equation}\label{eq:qq}
Q(s,r/q) = \frac{1}{2 \varphi (q)} \sum_{\chi \!\!\! \mod q} 
\bigl(1+\chi(-1)\bigr) \bigl( \overline{\chi} (r) q^s + G(r, \overline{\chi}) \bigl) L(s,\chi) .
\end{equation}
Therefore, we have the following from the functional equation (\ref{eq:zfe1q}). 
\begin{corollary}\label{cor:2}
The function $Q(2s,r/q)$ satisfies -(1)-, -(2)-, and
\begin{flushleft}
{\bf{-(3b')-}} There exists a positive integer $q$ such that $q^{-2s} \phi(s)$ can be expanded in a Dirichlet series that converges somewhere, and the only pole allowed for $\phi(s/2)$ is $s=1$. 
\end{flushleft}
\end{corollary}

For $q \in {\mathbb{N}}$, put $H(s,q) := (q^s+q^{1-s})^{-1}$. Then, we can see that $H(s,q)=H(1-s,q)$, and $q^s H(s,q)$ is written as an ordinary Dirichlet series by
\[
q^s H(s,q) = \frac{q^s}{q^s+q^{1-s}} = \frac{1}{1+q^{1-2s}} = \sum_{k=0}^\infty \frac{(-q)^k}{q^{2ks}}, \qquad \sigma >1/2. 
\]
From (\ref{eq:qq}), the function $q^{-s}Q(s,r/q)$ can be expressed as an ordinary Dirichlet series. Therefore, we can see that
\[
H(s,q) Q(s,r/q) = q^s H(s,q) \cdot q^{-s}Q(s,r/q)
\] 
is also written by an ordinary Dirichlet series. Moreover, the function
\[
(1-q^{1-s})(1+q^{1-2s})H(s,q) Q(s,r/q)
\]
is entire. Hence, we have the following from Theorem \ref{th:FE}. 
\begin{corollary}\label{cor:3}
The function $H(2s,q)Q(2s,r/q)$ satisfies (2), (3a), and
\begin{flushleft}
{\bf{-(1')-}} The function $\phi(s)$ is meromorphic and $D(s)\phi(s)$ is an entire function of finite genus with a suitable {\sc{Dirichlet}} polynomial $D(s)$. 
\end{flushleft}
\end{corollary}

Corollaries \ref{cor:1}, \ref{cor:2} and \ref{cor:3} can be regarded as analogues of Theorem C. Note that the functions appearing in Knopp's theorem have poles in the strip $1/2 -\sigma_0 \le \Re (s) \le \sigma_0$, where $\sigma_0 \ge 1/4$, from the condition -(3)-, but the zeta function $Q(2s,a)$ has only one pole at $s=1/2$ (see condition -(3b)- due to Hecke). Furthermore, the function $Q(2s,a)$ also fulfills the following splendid property by Theorem \ref{th:main}:
\begin{flushleft}
{\bf{-(0)-}} The function $\phi(s)$ has infinitely many zeros on the line $\sigma =1/4$. 
\end{flushleft}

Because Knopp did not explicitly provide solutions composed of zeta or $L$-functions that satisfy the conditions -(1)-, -(2)-, and -(3)-, we cannot see that the zeta or $L$-functions in his theorem fulfill any other noteworthy property. However, we show that the zeta function $Q(s,a)$ defined by (\ref{eq:defQSA}) satisfies the functional equation (H2) in Theorem \ref{th:FE}, or fulfills the condition -(2)- and some variations of -(1)- and -(3a)-, -(3b)-, or -(3)- in Corollaries \ref{cor:1}, \ref{cor:2}, and \ref{cor:3}. Furthermore, we prove that the zeta function $Q(s,a)$ has $\gg T$ zeros on the segment $(1/2, 1/2+iT)$ in Theorem \ref{th:main}. 

\section{Proofs}

\subsection{Functional equation and integral representation}
First, we prove Theorem \ref{th:FE}, namely, the functional equation (\ref{eq:zfe1q})
\begin{proof}[Proof of Theorem \ref{th:FE}]
For simplicity, we put
\[
Z(s,a) := \zeta (s,a) +  \zeta (s,1-a), \qquad P(s,a) := {\rm{Li}} (s,a) +  {\rm{Li}} (s,1-a) .
\]
Note that one has $2Q(s,a) = Z(s,a)+P(s,a)$. For $\sigma >1$, it is known that
\[
\zeta (1-s,a) =  \frac{2\Gamma (s)}{(2\pi)^s} \biggl\{ \cos \Bigl( \frac{\pi s}{2} \Bigr) 
\sum_{n=1}^\infty \frac{\cos 2\pi na}{n^s} + \sin \Bigl( \frac{\pi s}{2} \Bigr) 
\sum_{n=1}^\infty \frac{\sin 2\pi na}{n^s} \biggr\}
\]
(see for example \cite[Theorem 2.3.1]{LauGa}). Thus, we have
\[
Z (1-s,a) =  
\frac{4\Gamma (s)}{(2\pi)^s}  \cos \Bigl( \frac{\pi s}{2} \Bigr)\sum_{n=1}^\infty \frac{\cos 2\pi na}{n^s} 
= \frac{2\Gamma (s)}{(2\pi)^s}  \cos \Bigl( \frac{\pi s}{2} \Bigr) P(s,a) .
\]
In addition, for $\sigma <0$, it holds that
\[
{\rm{Li}} (1-s,a) = \frac{\Gamma (s)}{(2\pi )^s} \Bigl( e^{\pi is/2} \zeta (s,a) + e^{-\pi is/2} \zeta (s,1-a) \Bigr),
\qquad 0<a<1
\]
from \cite[Exercises 12.2 and 12.3]{Apo}. Hence, one has
\[
P(1-s,a) = \frac{2\Gamma (s)}{(2\pi )^s} \biggl( 
\cos \Bigl( \frac{\pi s}{2} \Bigr) \zeta (s,a) +\cos \Bigl( \frac{\pi s}{2} \Bigr)  \zeta (s,1-a) \biggr) 
=\frac{2\Gamma (s)}{(2\pi )^s} \cos \Bigl( \frac{\pi s}{2} \Bigr) Z(s,a) .
\]
We note that the functions $\zeta (s,a)$ and $Z(s,a)$ are regular for all $s \in {\mathbb{C}}$ except $s=1$ (see \cite[Chapter 12]{Apo}) and the functions ${\rm{Li}} (s,a)$ and $P(s,a)$ with $0<a<1$ are analytically continuable to an entire function (see \cite[Chapter 2.2]{LauGa}). Hence, the function $Q (s,a)$ is regular for all $s \in {\mathbb{C}}$ except $s=1$, where there is a simple pole with residue $1$. Therefore, by using the identity theorem and the functional equations of $Z(s,a)$ and $P(s,a)$ above, we have
\begin{equation*}
\begin{split}
&\quad \,\, 2Q(1-s,a) = Z(1-s,a)+P(1-s,a) \\ 
&=\frac{2\Gamma (s)}{(2\pi )^s} \cos \Bigl( \frac{\pi s}{2} \Bigr) \bigl( P(s,a) + Z(s,a) \bigr)
= \frac{2\Gamma (s)}{(2\pi )^s} \cos \Bigl( \frac{\pi s}{2} \Bigr) 2 Q(s,a)
\end{split}
\end{equation*}
which implies the functional equation of $Q(s,a)$.
\end{proof}

Next, we show the following integral representation of $\pi^{-s/2} \Gamma (s/2) Q(s,a)$. 
\begin{proposition}\label{lem:IR}
Define $G_{\! a}(u)$ by
\[
G_{\! a} (u) := \sum_{n \in \mathbb{Z}} \Bigl( \exp \bigl( -\pi u^2 (n+a)^2 \bigr) +
\exp \bigl( -\pi u^2 n^2 + i2\pi na \bigr) \Bigr).
\]
Then, for $0 < \Re (s) < 1$, we have
\[
\pi^{-s/2} \Gamma \Bigl( \frac{s}{2} \Bigr) Q(s,a) = \int_0^\infty u^{-s} \biggl( G_{\! a} (u) -1 -\frac{1}{u} \biggr) du.
\]
\end{proposition}

\begin{proof}
For $\Re (s) > 1$, we have
\begin{equation*}
\begin{split}
&\quad \,\, 2 \int_0^\infty u^{s-1} \sum_{n \in \mathbb{Z}} \exp \bigl( -\pi u^2 (n+a)^2 \bigr) du  \\ &=
2\sum_{n=0}^\infty \int_0^\infty u^{s-1} e^{ -\pi u^2 (n+a)^ 2 } du +
2\sum_{n=0}^\infty \int_0^\infty u^{s-1} e^{-\pi u^2 (n+1-a)^ 2 } du  
\end{split}
\end{equation*}
The first infinite series can be written as
\[
\sum_{n=0}^\infty \int_0^\infty e^{-v} \biggl( \frac{v/\pi}{(n+a)^ 2 } \biggr)^{s/2-1} \!\! \frac{dv /\pi}{(n+a)^2} \\
=  \pi^{-s/2} \Gamma \Bigl( \frac{s}{2} \Bigr) \sum_{n=0}^\infty (n+a)^{-s}.
\]
Hence, we obtain 
\[
2 \int_0^\infty u^{s-1} \sum_{n \in \mathbb{Z}} \exp \bigl( -\pi u^2 (n+a)^2 \bigr) du
= \pi^{-s/2} \Gamma \Bigl( \frac{s}{2} \Bigr) \bigl( \zeta (s,a) +  \zeta (s,1-a) \bigl). 
\]
Similarly, when $\Re (s) > 1$, one has
\begin{equation*}
\begin{split}
&\quad \,\, 2 \int_0^\infty u^{s-1} \sum_{0 \ne n \in \mathbb{Z}} \exp \bigl( -\pi u^2 n^2 + i2\pi na \bigr) du  \\ &=
2\sum_{n=0}^\infty \int_0^\infty u^{s-1} e^{2\pi i n a -\pi u^2 n^2}  du +
2\sum_{n=0}^\infty \int_0^\infty u^{s-1} e^{2\pi i n (1-a) -\pi u^2 n^2 } .
\end{split}
\end{equation*}
The first infinite series can be expressed as
\[
\sum_{n=0}^\infty \int_0^\infty e^{2\pi i n a} e^{-v} \biggl( \frac{v/\pi}{n^2 } \biggr)^{s/2-1} \frac{dv /\pi}{n^2} =
\pi^{-s/2} \Gamma \Bigl( \frac{s}{2} \Bigr) \sum_{n=0}^\infty \frac{e^{2\pi i n a}}{n^s} .
\]
Thus, it holds that
\[
2 \int_0^\infty \!\! u^{s-1} \!\! \sum_{0 \ne n \in \mathbb{Z}} \!\! \exp \bigl( -\pi u^2 n^2 + i2\pi na \bigr) du 
= \pi^{-s/2} \Gamma \Bigl( \frac{s}{2} \Bigr) \bigl( {\rm{Li}} (s,a) +  {\rm{Li}} (s,1-a) \bigl). 
\]
Therefore, when $\Re (s) >1$, we have
\[
\pi^{-s/2} \Gamma \Bigl( \frac{s}{2} \Bigr) Q(s,a) = \int_0^\infty u^{s-1} \bigl(  G_{\! a} (u) -1 \bigr) du .
\]
For $a,u>0$, it is well-known that (see \cite[p.~13, (6)]{KV})
\begin{equation*}
\begin{split}
&\sum_{n \in \mathbb{Z}} \exp \bigl( -\pi u^2 (n+a)^2 \bigr) =
\frac{1}{u} \sum_{n \in \mathbb{Z}} \exp \bigl( -\pi u^{-2} n^2 + i2\pi na \bigr), \\
&\sum_{n \in \mathbb{Z}} \exp \bigl( -\pi u^2 n^2 + i2\pi na \bigr) =
\frac{1}{u} \sum_{n \in \mathbb{Z}} \exp \bigl( -\pi u^{-2} (n+a)^2 \bigr) .
\end{split}
\end{equation*}
Hence, we easily obtain
\begin{equation}\label{eq:G}
G_{\! a} (u) = u^{-1} G_{\! a} (u^{-1}), \qquad u >0. 
\end{equation}
By using the equation above and changing the variable $u \to v^{-1}$, we have
\[
\pi^{-s/2} \Gamma \Bigl( \frac{s}{2} \Bigr) Q(s,a) = \int_0^\infty v^{-s} \biggl( G_{\! a}(v) -\frac{1}{v} \biggr) dv 
\]
when $\Re (s) >1$. Hence, we have
\begin{equation}
\begin{split}\label{eq:2int}
&\pi^{-s/2} \Gamma \Bigl( \frac{s}{2} \Bigr) Q(s,a) = 
\int_0^1 \!\! v^{-s} \biggl( G_{\! a} (v) -\frac{1}{v} \biggr) dv + \int_1^\infty \!\!\! v^{-s} \biggl( G_{\! a}(v) -\frac{1}{v} \biggr) dv \\
&= \int_0^1 v^{-s} \biggl( G_{\! a} (v) -\frac{1}{v} \biggr) dv + 
\int_1^\infty v^{-s} \biggl( G_{\! a} (v) -1 -\frac{1}{v} \biggr) dv  + \frac{1}{s-1} .
\end{split}
\end{equation}
Let $a_\star := \min \{a,1-a \}$. From the definition of $G_{\! a} (u) $ and (\ref{eq:G}), it holds that
\[
G_{\! a} (u) = \frac{1}{u} \Bigl( 1 + O \bigl( \exp(- a_\star^2 u^{-2}) \bigl) \Bigr), \qquad u \to 0.
\]
Let $\Re (s) > 0$. By the definition of $G_{\! a} (u) $ and the estimation above, we have
\begin{equation*}
\begin{split}
&\int_0^1 v^{-s} \biggl( G_{\! a} (v) - \frac{1}{v} \biggr) dv \ll \int_0^1 v^{-\sigma-1} \exp \bigl(- a_\star^2 v^{-2} \bigl) dv
< \infty, \\
&\int_1^\infty v^{-s} \biggl( G_{\! a}(v) -1 -\frac{1}{v} \biggr) dv \ll 
\int_1^\infty v^{-\sigma-1} dv + \int_1^\infty \biggl( G_{\! a} (v) - 1 \biggr) dv \\ & \ll 
\int_1^\infty \sum_{n=1}^\infty \exp \bigl( -\pi v n^2 \bigr) dv < \infty.
\end{split}
\end{equation*}
Hence, both integrals in (\ref{eq:2int}) converge when $\Re (s) > 0$. Clearly, one has 
\[
\frac{1}{s-1} = - \int_0^1 v^{-s} dv
\] 
for $\Re(s) <1$. Therefore, we obtain the integral representation in Proposition \ref{lem:IR}
\end{proof}

\subsection{Lemmas}
We contrast the integrals $J_a (t)$ and $I_a (t)$ given by
\begin{equation*}
\begin{split}
J_a (t)  &:= \frac{1}{2\pi} \int_{t-k}^{t+k} \biggl|
\pi^{-1/4-iu/4} \Gamma \Bigl( \frac{1 + 2iu}{4} \Bigr) Q(1/2 + iu,a) e^{\pi u/4} e^{-u\delta/2} \biggr| du \\
I_a (t) &:=  \frac{e^{i \pi /8} e^{-i \delta/4}}{2\pi} \int_{t-k}^{t+k} 
\pi^{-1/4-iu/4} \Gamma \Bigl( \frac{1+2iu}{4} \Bigr) Q(1/2+iu,a) e^{\pi u/4} e^{-u\delta/2} du 
\end{split}
\end{equation*}
to show Theorem \ref{th:main}. This idea is parallel to the proof in \cite[Chapter 11.2]{ERZ} or \cite[Chapter 10.7]{Tit}. Note that $|J_a (t)| > |I_a (t)|$ if the interval of integration of $I_a (t)$ contains roots of $Q(s,a)=0$ on $\Re(s) =1/2$ and $|J_a (t)| = |I_a (t)|$ otherwise.  The key to the proof of Theorem \ref{th:main} is to estimate the total length of all the integral intervals satisfying $|J_a (t)| > |I_a (t)|$. For this purpose, we present some lemmas. 
\begin{lemma}\label{lem:Inlow}
Let $a \ne 1/4$, $t$ and $k$ be sufficiently large. Then, we have
\begin{equation}
\begin{split}
&\int_{t-k}^{t+k} \bigl| Q(1/2+iv, a) \bigr| dv \ge 2k |\cos (2\pi a)| - C_1 (a) - \frac{C_2 (a) k^2}{(t-k)^{1/2}} \\ &
- \Biggl| \sum_{a_*=a, 1-a} \sum_{2 \le n < t/a_\star} \!\!
\biggl( \frac{\sin (k \log (n + a_*))}{(n \!+\! a_*)^{1/2+it}\log(n \!+\! a_*)} + 
\frac{e^{2\pi i a_* n} \sin (k \log n)}{n^{1/2+it}\log n} \biggr) \Biggr|,
\end{split}
\end{equation}
where $C_1(a)$ and $C_2(a)$ are some positive constants depend on $a$. 
\end{lemma}
\begin{proof}
For $a_*=a$ or $1-a$, the following approximate functional equations are shown (see \cite[Theorems 3.1.3 and 3.1.2]{LauGa}). Suppose $2\pi \le |\tau| \le \pi x$. Then, one has 
\[
\zeta(1/2+i\tau,a_*) = \sum_{n=0}^x \frac{1}{(n+a_*)^{1/2+i\tau}} + O \bigl( x^{-1/2} \bigr) .
\]
Moreover, for $|t| \le \pi a_* x$, it holds that
\[
{\rm{Li}} (1/2+i\tau,a_*) = \sum_{n=1}^x \frac{e^{2\pi i a_* n}}{n^{1/2+i\tau}} + O \bigl( x^{-1/2} \bigr) .
\]
Let $a_\star := \min \{a,1-a\}$ again. Then, for $t-k \le v \le t+k$, we have
\begin{equation*}
\begin{split}
&\zeta(1/2+iv,a_*) = \sum_{0 \le n < t/a_\star} \frac{1}{(n+a_*)^{1/2+iv}} + E_\zeta (v,a_*) + O \bigl( v^{-1/2} \bigr), \\
&{\rm{Li}} (1/2+iv,a_*) = \sum_{1 \le n < t/a_\star} \frac{e^{2\pi i a_* n}}{n^{1/2+iv}} + E_{\rm{Li}}(v,a_*)
+ O \bigl( v^{-1/2} \bigr),
\end{split}
\end{equation*}
where $E_\zeta (v,a_*)$ is defined by
\[
E_\zeta (v,a_*) :=
\begin{cases}
\,\, \sum_{t/a_\star \le n < v/a_\star} (n+a_*)^{-1/2-iv} &  t \le v, \\
-\sum_{v/a_\star \le n < t/a_\star} (n+a_*)^{-1/2-iv} & t > v.
\end{cases}
\]
The function $E_{\rm{Li}} (v,a_*)$ is defined similarly. Because $E_\zeta (v,a_*)$ consists of at most $2k/a_\star$ terms each of modulus at most $(n+a_*)^{-1/2} \le a_\star^{-1/2} (t-k)^{-1/2}$, we obtain that for $t-k \le v \le t+k$, 
\[
\Re \Biggl( \sum_{0 \le n < t/a_\star} \frac{1}{(n+a_*)^{1/2+iv}} - \zeta(1/2+iv,a_*) \Biggr) \le 
\frac{2k a_\star^{-3/2}}{(t-k)^{1/2}} + O \bigl( v^{-1/2} \bigr) \le \frac{C_2'(a_*) k}{(t-k)^{1/2}}
\]
for some positive constant $C_2'(a_*)$. Similarly, for some constant $C_2''(a_*)>0$, it holds that
\[
\Re \Biggl( \sum_{1 \le n < t/a_\star} \! \frac{e^{2\pi i a_* n}}{n^{1/2+iv}} - {\rm{Li}} (1/2+iv,a_*) \Biggr) \le 
\frac{2k a_\star^{-3/2}}{(t-k)^{1/2}} + O \bigl( v^{-1/2} \bigr) \le \frac{C_2''(a_*)k}{(t-k)^{1/2}}.
\]
Assume $0<a< 1/4$ and put $M(k,a) := 4k \cos (2\pi a) - 2 C_1 (a)$ and
\[
M (k,t,a) := 4k \cos (2\pi a) - 2 C_1 (a) - \frac{2C_2 (a) k^2}{(t-k)^{1/2}},
\]
where $0 < C_2'(a)  + C_2''(a) + C_2'(1-a)  + C_2''(1-a) \le C_2 (a)$. For simplicity, we put $\sum_{*,\star} : = \sum_{a_*=a, 1-a} \sum_{2 \le n < t/a_\star}$. Then, we have
\begin{equation*}
\begin{split}
&2 \int_{t-k}^{t+k} \bigl| Q(1/2+iv, a) \bigr| dv \ge 2 \int_{t-k}^{t+k} \Re \bigl( Q(1/2+iv, a) \bigr) dv \\ \ge
& \, M (k,a) + \int_{t-k}^{t+k}  \Biggl( \Re \sum_{*,\star} 
\biggl( \frac{1}{(n+a_*)^{1/2+iv}} + \frac{e^{2\pi i a_* n}}{n^{1/2+iv}} \biggr) - \frac{C_2 (a) k}{(t-k)^{1/2}} \Biggr) dv \\
 = & \, M (k,t,a) + \Re \sum_{\smash{*,\star}} 
\biggl( \frac{2 \sin (k \log (n+a_*))}{(n+a_*)^{1/2+it}\log(n+a_*)} + 
\frac{2e^{2\pi i a_* n} \sin (k \log n)}{n^{1/2+it}\log n} \biggr) 
\end{split}
\end{equation*}
which implies Lemma \ref{lem:Inlow} with $0<a< 1/4$. When $1/4<a< 1/2$, based on the inequality
\[
\int_{t-k}^{t+k} \bigl| Q(1/2+iv, a) \bigr| dv \ge - \int_{t-k}^{t+k} \Re \bigl( Q(1/2+iv, a) \bigr) dv,
\]
we obtain Lemma \ref{lem:Inlow} with $1/4 < a < 1/2$.
\end{proof}

\begin{lemma}\label{lem:C-S}
Put $a_\star := \min \{a,1-a\}$ and $a_* = a$ or $1-a$. Let $A \le t \le B$ with $B \ge A \ge 2/a_\star$. Then, it holds that
\begin{equation*}
\begin{split}
& \int_A^B \Biggl| \sum_{2 \le n <t/a_\star} \frac{\sin (k \log (n+a_*))}{(n+a_*)^{1/2+it} \log(n+a_*)} \Biggr| dt \le C_3 (a_*) B, \\
&\int_A^B \Biggl| \sum_{2 \le n <t/a_\star} \frac{e^{2\pi ina_*} \sin (k \log n)}{n^{1/2+it} \log n} \Biggr| dt \le C_4 (a_*) B,
\end{split}
\end{equation*}
for some positive constants $C_3 (a_*)$ and $C_4 (a_*)$. 
\end{lemma}

\begin{proof}
First, we estimate the integral
\begin{equation*}
\begin{split}
&\int_A^B \Biggl| \sum_{2 \le n < t/a_\star} \frac{\sin (k \log (n+a_*))}{(n+a_*)^{1/2+it} \log(n+a_*)} \Biggr|^2 dt = \\
&\int_A^B \sum_{2 \le n,m < t/a_\star} 
\frac{\sin (k \log (n+a_*)) \sin (k \log (m+a_*))}{(n+a_*)^{1/2} (m+a_*)^{1/2} \log(n+a_*) \log(m+a_*)} 
\Bigl( \frac{m+a_*}{n+a_*} \Bigr)^{it} dt.
\end{split}
\end{equation*}
For the terms with $n = m$, we have
\[
\int_A^B \sum_{2 \le n < t/a_\star} \frac{\sin^2 (k \log (n+a_*))}{(n+a_*) (\log(n+a_*))^2} dt 
\ll \int_A^B \sum_{2 \le n < t/a_\star} \frac{1}{n (\log n) ^2} dt.
\]
It should be noted that the infinite series $\sum_{n=2}^\infty n^{-1} (\log n)^{-2}$ converges. Each of the terms with $n \ne m$ is of the form
\[
\frac{\sin (k \log (n+a_*))\sin (k \log (m+a_*))}{(n+a_*)^{1/2} (m+a_*)^{1/2} \log(n+a_*) \log(m+a_*)} 
\int_b^B \Bigl( \frac{m+a_*}{n+a_*} \Bigr)^{it} dt,
\]
where $b: = \max (A, a_\star m, a_\star n)$. Thus, regardless of the value of $b$, its absolute value is bounded above by
\[
\frac{2(n+a_*)^{-1/2} (m+a_*)^{-1/2} }{\log(n+a_*) \log(m+a_*) |\log((n+a_*)/(m+a_*))|} \ll
\frac{n^{-1/2} m^{-1/2} }{\log n \log m |\log(n/m)|}.
\]
It is shown in \cite[p.~236]{ERZ} that
\begin{equation}\label{eq:iq236}
\sum_{2 \le n \ne m \le B} \frac{n^{-1/2} m^{-1/2} }{\log n \log m |\log(n/m)|} \ll B.
\end{equation}
Therefore, it holds that 
\[
\int_A^B \Biggl| \sum_{2 \le n < t/a_\star} \frac{\sin (k \log (n+a_*))}{(n+a_*)^{1/2+it} \log(n+a_*)} \Biggr|^2  dt \ll B.
\]
From the inequality above and the the Cauchy-Schwarz inequality, we have
\[
\int_A^B \Biggl| \sum_{2 \le n < t/a_\star} \frac{\sin (k \log (n+a_*))}{(n+a_*)^{1/2+it} \log(n+a_*)} \Biggr| dt 
\ll (B-A)^{1/2} B^{1/2} \ll B,
\]
which implies the first inequality of Lemma \ref{lem:C-S}. 

Second, consider the integral
\begin{equation*}
\begin{split}
&\int_A^B \Biggl| \sum_{2 \le n < t/a_\star} \frac{e^{2\pi ina_*} \sin (k \log n)}{n^{1/2+it} \log n} \Biggr|^2 dt = \\ &\int_A^B 
\sum_{2 \le n,m  < t/a_\star} \frac{e^{2\pi ina_*} e^{-2\pi ima_*} \sin (k \log n) \sin (k \log m)}{n^{1/2} m^{1/2} \log n \log m} 
\Bigl( \frac{m}{n} \Bigr)^{it} dt.
\end{split}
\end{equation*}
Obviously, for the terms with $n = m$, we have
\[
\int_A^B \sum_{2 \le n < t/a_\star} \frac{e^{2\pi ina_*} e^{-2\pi ina_*} \sin^2 (k \log n)}{n (\log n)^2} dt \le
\int_A^B \sum_{2 \le n < t/a_\star} \frac{1}{n (\log n) ^2} dt.
\]
For each of the terms with $n \ne m$, we have
\[
\biggl| \int_b^B \frac{e^{2\pi ina_*} e^{-2\pi ima_*} \sin (k \log n) \sin (k \log m)}{n^{1/2} m^{1/2} \log n \log m} \Bigl( \frac{m}{n} \Bigr)^{it} dt \biggr| \ll \frac{n^{-1/2} m^{-1/2} }{\log n \log m |\log(n/m)|}.
\]
Thus, we have the second inequality of Lemma \ref{lem:C-S} from (\ref{eq:iq236}) and the Cauchy-Schwarz inequality. 
\end{proof}

\subsection{Proof of the existence of zeros on the critical line}
Let $k$ be a positive real number, $x$ be a complex number with $|x|=1$ and $|\Im (\log x)| \le \pi /4$, and put
\begin{equation}\label{eq:IIsa(1)}
I_{x,k} (s,a) := \frac{1}{2\pi i} \int_{s-ik}^{s+ik} \pi^{-v/2} \Gamma \Bigl( \frac{v}{2} \Bigr) Q(v,a) x^{v-1}dv.
\end{equation}
Then, by Proposition \ref{lem:IR}, the function $I_{x,k} (s,a)$ can be expressed as
\begin{equation*}
\begin{split}
&\frac{1}{2\pi i} \! \int_{s-ik}^{s+ik} \! \int_0^\infty \Bigl( \frac{u}{x} \Bigr)^{-v} 
\biggl( G_{\! a} (u) - 1 - \frac{1}{u} \biggr) \frac{dudv}{x} \\ = & \, 
\frac{1}{2\pi i} \! \int_{s-ik}^{s+ik} \! \int_0^{\infty/x} \! \biggl( G_{\! a} (xw) - 1 - \frac{1}{xw} \biggr) \frac{dwdv}{w^{v}} .
\end{split}
\end{equation*}
By Cauchy's integral theorem and the fact that the function $G_{\! a}(u)-1$ approaches zero rapidly as $u$ tends to infinity along any ray $u=xw$ in the wedge $|\Im (\log x)| \le \pi /4$, $w \in {\mathbb{R}}$, the integral above is equal to
\begin{equation*}
\begin{split}
&\frac{1}{2\pi i} \int_{s-ik}^{s+ik} \int_0^{\infty} w^{-v} \biggl( G_{\! a} (xw) - 1 - \frac{1}{xw} \biggr) dwdv \\ &=
\int_0^{\infty} \biggl( \frac{1}{2\pi i} \int_{s-ik}^{s+ik} w^{-v} dv \biggr)
\biggl( G_{\! a} (xw) - 1 - \frac{1}{xw} \biggr) dw 
\\ &= \frac{1}{\pi} \int_0^{\infty} \frac{w^{-s} \sin (k \log w)}{\log w} \biggl( G_{\! a} (xw) - 1 - \frac{1}{xw} \biggr)dw .
\end{split}
\end{equation*}
This expresses $I_{x,k} (s,a)$ as the transform of an operator and shows, from the Parseval-Plancherel identity (see \cite[p.~216, line 7]{ERZ}), that
\begin{equation}\label{eq:IIsa}
\frac{1}{2\pi i} \int_{1/2-i\infty}^{1/2+i\infty} \bigl| I_{x,k} (s,a) \bigr|^2 ds = 
\frac{1}{\pi^2} \int_0^{\infty} \biggl| \frac{\sin (k \log w)}{\log w} \biggr|^2 
\biggl| G_{\! a}(xw) - 1 - \frac{1}{xw} \biggr|^2 dw .
\end{equation}

Note that under the change of variable $w \to w^{-1}$, the form $dw$ becomes $-dw/w^2$, the factor $\sin (k \log w)/\log w$ is unchanged, and the function $G_{\! a}(xw) - 1 - (xw)^{-1}$ becomes
\[
G_{\! a} \Bigl( \frac{1}{\overline{x} w} \Bigr) -1 - \overline{x} w = 
\overline{x} w \biggl( \frac{1}{\overline{x} w} 
G_{\! a} \Bigl( \frac{1}{\overline{x} w} \Bigr) - \frac{1}{\overline{x} w} -1 \biggr)
= \overline{x} w \biggl( G_{\! a} ( \overline{x} w) - 1 - \frac{1}{\overline{x} w} \biggr),
\]
where $x^{-1} = \overline{x}$, from (\ref{eq:G}). Thus, the integral of the right-hand side of (\ref{eq:IIsa}) is equal to twice the integral from $1$ to $\infty$. The first step is deriving an upper bound of the integral given by (\ref{eq:IIsa}). 

\begin{lemma}\label{lem:upper}
Let $I_{x,k} (s,a)$ be defined as in (\ref{eq:IIsa(1)}) with $x= e^{-i \pi /4} e^{i\delta/2}$. Then, there exists a constant $K >0$ such that given $\varepsilon >0$ the inequality
\[
\frac{1}{2\pi i} \int_{1/2-i\infty}^{1/2+i\infty} \bigl| I_{x,k} (s,a) \bigr|^2 ds < \frac{Kk+\varepsilon k^2}{\delta^{1/2}}
\]
holds for all sufficiently small $\delta>0$ (with $k>0$ being arbitrary).
\end{lemma}

\begin{proof}
From the inequality
\[
\biggl| \frac{\sin ky}{y} \biggr| \le
\begin{cases}
k & 0 \le \pi /k, \\
y^{-1} & y \ge \pi / k,
\end{cases}
\]
the integral of the left-hand side of (\ref{eq:IIsa}) is bounded above by
\begin{equation}\label{eq:IIsa(3)}
\frac{2k^2}{\pi^2} \!\! \int_1^{e^{\pi/k}} \biggl| G_{\! a} (xw) - 1 - \frac{1}{xw} \biggr|^2 dw 
+ \frac{2}{\pi^2} \! \int_{e^{\pi/k}}^{\infty} \biggl| G_{\! a} (xw) - 1 - \frac{1}{xw} \biggr|^2 \frac{dw}{(\log w)^{2}}.
\end{equation}
According to the inequality $|A+B|^2 \le 2|A|^2+2|B|^2$, where $A,B \in {\mathbb{C}}$, the first integral in (\ref{eq:IIsa(3)}) is at most
\[
\frac{4k^2}{\pi^2} \int_1^{e^{\pi/k}} \bigl| G_{\! a} (xw) - 1  \bigr|^2 dw + 
\frac{4k^2}{\pi^2} \int_1^{e^{\pi/k}} \frac{dw}{w^2}.
\]
Obviously, the second definite integral is $4k^2 \pi^{-2}(1-e^{-\pi /k})$. From $|A+B|^2 \le 2|A|^2+2|B|^2$ again, the first integral is bounded above by
\begin{equation*}
\begin{split}
\frac{8 k^2}{\pi^2} \!\! \int_1^{e^{\pi/k}} \Biggl| \sum_{n \in \mathbb{Z}} \exp \bigl( -\pi x^2 w ^2 (n+a)^2 \bigr) \Biggr|^2 dw 
+ \frac{8 k^2}{\pi^2} \!\! \int_1^{e^{\pi/k}} \Biggl|\sum_{0 \ne n \in \mathbb{Z}} \exp \bigl( -\pi x^2 w^2 n^2 + i2\pi na \bigr) \Biggr|^2 dw
\end{split}
\end{equation*}
Let $x=e^{i\pi/4} e^{-i \delta /2}$ so that $x^2 = \sin \delta + i \cos \delta$. Then, the integrand in the first and second integrals respectively become
\begin{equation}\label{eq:(4)}
\sum_{a_\flat, \, a_\sharp \,} \sum_{n,m =0}^\infty \exp \bigl( G_{n,m}^{a_\flat, a_\sharp} (w, \delta) \bigr), 
\qquad 
\sum_{a_\flat, \, a_\sharp \,} \sum_{n,m =1}^\infty \exp \bigl( H_{n,m}^{a_\flat, a_\sharp} (w, \delta) \bigr), 
\end{equation}
where the sum $\sum_{a_\flat, \, a_\sharp}$ taken over $(a_\flat, a_\sharp) =(a,a), (a,1-a), (1-a,a), (1-a, 1-a)$ and the functions $G_{n,m}^{a_\flat, a_\sharp} (w, \delta)$ and $G_{n,m}^{a_\flat, a_\sharp} (w, \delta) $ are defined by
\begin{equation*}
\begin{split}
&G_{n,m}^{a_\flat, a_\sharp} (w, \delta) := 
-\pi \bigl( (n+a_\flat)^2 + (m+a_\sharp)^2 \bigr) w^2 \sin \delta 
-i\pi \bigl( (n+a_\flat)^2 - (m+a_\sharp)^2 \bigr) w^2 \cos \delta , \\
&H_{n,m}^{a_\flat, a_\sharp} (w, \delta) := 
-\pi \bigl( n^2 + m^2 \bigr) w^2 \sin \delta -i\pi \bigl( n^2 - m^2 \bigr) w^2 \cos \delta 
+ i2\pi na_\flat - i2\pi ma_\sharp, 
\end{split}
\end{equation*}
respectively. We divide each double sum $\sum_{n,m =0}^\infty$ and $\sum_{n,m =1}^\infty$ into three sums, one in which $n=m$, one in which $n>m$ and one in which $n<m$ because both double sums in (\ref{eq:(4)}) converge absolutely. If $n=m$, we have
\begin{equation}\label{eq:C1a}
\begin{split}
&\sum_{a_\flat, \, a_\sharp \,}  \Biggl( \sum_{n=0}^\infty \exp \bigl( G_{n,n}^{a_\flat, a_\sharp} (w, \delta) \bigr) 
+ \sum_{n=1}^\infty \exp \bigl( H_{n,n}^{a_\flat, a_\sharp} (w, \delta) \bigr) \Biggr)
\\ & \le C_5 (a) \sum_{n=1}^\infty \exp \Bigl( -2\pi n^2 w^2 \sin \delta \Bigr)
\end{split}
\end{equation}
for some positive constant $C_5 (a)$. From the inequality
\[
\frac{8 k^2}{\pi^2} \int_1^{e^{\pi/k}} \sum_{n=1}^\infty \exp \Bigl( -2\pi n^2 w^2 \sin \delta \Bigr) dw \ll 
\frac{k^2}{\pi^2} \int_1^{e^{\pi/k}} \!\!\! \frac{dw}{w (2\sin \delta)^{1/2}}
\ll k \delta^{-1/2}, 
\]
which is proved in \cite[p.~231, line 7 from the bottom]{ERZ}, for some positive constant $C_6 (a)$, it holds that 
\begin{equation*}
\begin{split}
&\frac{8 k^2}{\pi^2} \int_1^{e^{\pi/k}} \sum_{a_\flat, \, a_\sharp \,}  \Biggl( \sum_{n=0}^\infty
\exp \bigl( G_{n,n}^{a_\flat, a_\sharp} (w, \delta) \bigr) +  \sum_{n=1}^\infty 
\exp \bigl( H_{n,n}^{a_\flat, a_\sharp} (w, \delta) \bigr) \Biggr)
dw \\ & \le C_6(a) k \delta^{-1/2}. 
\end{split}
\end{equation*}

Next, we will estimate the definite integral of the remaining terms $n \ne m$ of (\ref{eq:(4)}) from $1$ to $e^{\pi/k}$. The terms with $m > n$ are the complex conjugates of those with $m < n$, so it suffices to estimate the latter. Consider the integral 
\[
\frac{8 k^2}{\pi^2} \!\! \int_1^{e^{\pi/k}} \!\! \!\!\!\! \!\! \exp
\bigl( -\pi \bigl( (n+a)^2 + (m+a)^2 \bigr) w^2 \sin \delta -i\pi \bigl( (n+a)^2 - (m+a)^2 \bigr) w^2 \cos \delta \bigr) dw
\]
when $n>m$. The real part of the integral is
\[
\frac{8 k^2}{\pi^2} \int_1^{e^{\pi/k}} \!\! f(w,a) \cos V(w,a) dw, 
\]
\begin{equation*}
\begin{split}
\mbox{where} \qquad \qquad \quad
\begin{cases}
f(w,a) := \exp \bigl( -\pi \bigl( (n+a)^2 + (m+a)^2 \bigr) w^2 \sin \delta \bigr), \\
V(w,a) := \pi \bigl( (n+a)^2 - (m+a)^2 \bigr) w^2 \cos \delta . \qquad \qquad \quad
\end{cases}
\end{split}
\end{equation*}
Because $\cos \delta$ is positive for a small $\delta >0$, $V(w,a)$ is a monotone increasing function with respect to $w$, and this  integral can be rewritten in terms of the variable $V$ as
\[
\frac{8 k^2}{\pi^2} \int_{V(1)}^{V(e^{\pi/k})} \frac{f}{V'} \cos V dV,
\]
where $f$ and $V'$ are functions of $V$ by composition with the inverse function $V \to w$. By \cite[Lemma in p.~197]{ERZ} and the fact that $f$ is decreasing and $V'$ is increasing, the integral above is not more than
\[
\frac{8 k^2}{\pi^2} \frac{2f(1)}{V'(1)} = \frac{16 k^2}{\pi^2} 
\frac{\exp \bigl( -\pi \bigl( (n+a)^2 + (m+a)^2 \bigr) w^2 \sin \delta \bigr)}{2 \pi \bigl( (n+a)^2 - (m+a)^2 \bigr) \cos \delta}.
\]
A similar estimate can be applied to the imaginary part and to both the real and imaginary parts of the following integral 
\[
\frac{8 k^2}{\pi^2} \!\! \int_1^{e^{\pi/k}} \!\! \!\!\!\! \!\! \exp 
\bigl( -\pi \bigl( n^2 + m^2 \bigr) w^2 \sin \delta -i\pi \bigl( n^2 - m^2 \bigr) w^2 \cos \delta 
+ i2\pi na - i2\pi ma \bigr) dw .
\]
Hence, for some positive constant $C_7 (a)$ and $a_\flat, a_\sharp = a$ or $1-a$, we have
\begin{equation}\label{eq:C3a}
\begin{split}
&\frac{8 k^2}{\pi^2} \biggl| \int_1^{e^{\pi/k}} \exp \bigl( G_{n,m}^{a_\flat, a_\sharp} (w, \delta) \bigr) dw \biggr|, \qquad
\frac{8 k^2}{\pi^2} \biggl| \int_1^{e^{\pi/k}} \exp \bigl( H_{n,m}^{a_\flat, a_\sharp} (w, \delta) \bigr) dw \biggr|
\\ & \le C_7(a) k^2 \frac{\exp \bigl( -\pi (n^2 + m^2) \sin \delta \bigr)}{(n^2 - m^2) \cos \delta}.
\end{split}
\end{equation}
For any $\varepsilon >0$ and for all sufficiently small $\delta >0$, we have
\[
\sum_{n=1}^\infty \sum_{m=0}^{n-1}  \frac{k^2 \exp \bigl( -\pi (n^2 + m^2) \sin \delta \bigr)}{(n^2 - m^2)}
\ll k^2 \sum_{n=1}^\infty \exp (-n^2 \sin \delta) \frac{\log n}{n} 
\le \varepsilon k^2 \delta^{-1/2}. 
\]
from \cite[p.~233, line 1]{ERZ}. Therefore, it holds that
\[
\frac{8 k^2}{\pi^2} \int_1^{e^{\pi/k}} \sum_{a_\flat, \, a_\sharp \,} \sum_{n=1}^\infty \Biggl(
\sum_{m=0}^{n-1} \exp \bigl( G_{n,m}^{a_\flat, a_\sharp} (w, \delta) \bigr) + 
\sum_{m=1}^{n-1} \exp \bigl( H_{n,m}^{a_\flat, a_\sharp} (w, \delta) \bigr) \Biggr) dw \le \varepsilon k^2 \delta^{-1/2}. 
\]

Similar methods prove that the same estimation applies to the second integral in (\ref{eq:IIsa(3)}). We can easily see that
\[
\frac{2}{\pi^2} \int_{e^{\pi/k}}^{\infty} (\log w)^{-2} \biggl| \frac{1}{xw} \biggr|^2 dw \le 
\frac{2}{\pi^2} (\log e^{\pi/k} )^{-2} \int_{e^{\pi/k}}^{\infty} \frac{dw}{w^2} \le \varepsilon k^2 \delta^{-1/2}.
\]
It is shown in \cite[p.~233, line 15]{ERZ} that
\[
\int_{e^{\pi/k}}^{\infty} (\log w)^{-2} \sum_{n=1}^\infty \exp \Bigl( -2\pi n^2 w^2 \sin \delta \Bigr) \ll k \delta^{-1/2}.
\]
Hence, from (\ref{eq:C1a}), it holds that 
\[
\int_{e^{\pi/k}}^{\infty} (\log w)^{-2} \!\! \sum_{a_\flat, \, a_\sharp \,}  \Biggl( \sum_{n=0}^\infty
\exp \bigl( G_{n,n}^{a_\flat, a_\sharp} (w, \delta) \bigr) + 
\sum_{n=1}^\infty \exp \bigl( H_{n,n}^{a_\flat, a_\sharp} (w, \delta) \bigr) \Biggr) dw \ll k \delta^{-1/2}.
\]
Moreover, it is shown in \cite[p.~233, line 12 from the bottom]{ERZ} that
\begin{equation*}
\begin{split}
&\int_{e^{\pi/k}}^{\infty}  \sum_{n=1}^\infty \sum_{m<n} 
\exp \Bigl( -\pi (n^2+m^2) w^2 \sin \delta -i\pi (n^2-m^2) w^2 \cos \delta \Bigr) \frac{dw}{(\log w)^2} \\
&\ll \sum_{n=1}^\infty \sum_{m<n}  \frac{\exp \bigl( -\pi (n^2 + m^2) \sin \delta \bigr)}{(\log e^{\pi/k})^2(n^2 - m^2)}
\le \varepsilon k^2 \delta^{-1/2}.
\end{split}
\end{equation*}
Thus, by the inequality above and modifying the proof of (\ref{eq:C3a}), we have
\[
\sum_{a_\flat, \, a_\sharp \,} \sum_{n=1}^\infty \int_{e^{\pi/k}}^{\infty} \Biggl( \sum_{m=0}^{n-1} 
\exp \bigl( G_{n,m}^{a_\flat, a_\sharp} (w, \delta) \bigr) + \sum_{m=1}^{n-1} 
\exp \bigl( H_{n,m}^{a_\flat, a_\sharp} (w, \delta) \bigr) \Biggr)
\frac{dw}{(\log w)^2} \le \varepsilon k^2 \delta^{-1/2}.
\]
The estimate of the integral (\ref{eq:IIsa}) is unchanged under $e^{i\pi/4} e^{-i \delta /2} \to e^{-i\pi/4} e^{i \delta /2}$ by the complex conjugate of $I_{x,k}(s,a)$. Hence, Lemma \ref{lem:upper} is obtained. 
\end{proof}

Next, we estimate the integral
\[
J_{k} (t,a) := \frac{1}{2\pi} \int_{t-k}^{t+k} \biggl|
\pi^{-1/4-iu/4} \Gamma \Bigl( \frac{1 + 2iu}{4} \Bigr) Q(1/2 + iu,a) e^{\pi u/4} e^{-u\delta/2} \biggr| du.
\]
When $x = e^{-i\pi/4} e^{i \delta /2}$, we have
\begin{equation*}
\begin{split}
&I_{x,k} (1/2+it,a) = \frac{1}{2\pi} \int_{t-k}^{t+k} 
\pi^{-1/4-iu/4} \Gamma \Bigl( \frac{1+2iu}{4} \Bigr) Q(1/2+iu,a) x^{-1/2} x^{iu} du \\
&=  \frac{x^{-1/2}}{2\pi} \int_{t-k}^{t+k} 
\pi^{-1/4-iu/4} \Gamma \Bigl( \frac{1+2iu}{4} \Bigr) Q(1/2+iu,a) e^{\pi u/4} e^{-u\delta/2} du .
\end{split}
\end{equation*}
For simplicity, we put
\begin{equation*}\label{eq:defIJ}
I_a (t) := I_{x,k} (1/2+it,a), \qquad J_a(t) := J_{k} (t,a).
\end{equation*}
Then, we have $J_a (t) \ge |I_a(t)|$ for all $t \in {\mathbb{R}}$ and $J_a (t) =|I_a (t)|$ whenever the interval of integration of $I_a (t)$ contains no roots of $Q(s,a)=0$ on the line $\Re(s) =1/2$. The basic idea of the proof is to show that in a suitable sense $J_a (t)$ is much larger than $|I_a (t)|$ on average. Thus, estimates of $J_a (t)$ from below are required. Stirling's formula yields 
\begin{equation*}
\begin{split}
\Gamma \Bigl( \frac{1+2iu}{4} \Bigr) \bigl| Q(1/2+iu,a) \bigr| e^{-\pi u/4} &\gg 
u^{3/4} e^{-u \pi /4} u^{-1} \bigl| Q(1/2+iu,a) \bigr| e^{\pi u/4} \\ & \gg u^{-1/4} \bigl| Q(1/2+iu,a) \bigr| 
\end{split}
\end{equation*}
when $u$ is sufficiently large. Thus, we have the following:
\begin{lemma}\label{lem:Jest}
For sufficiently large $t>0$, it holds that
\[
J_a (t) \gg (t+k)^{-1/4} e^{-(t+k)\delta/2} \int_{t-k}^{t+k} \bigl| Q(1/2+iu,a) \bigr| du.
\]
\end{lemma}

Now, we are in a position to prove the main theorem. Note that the proof below is based on the argument in \cite[Chapter 11.2]{ERZ} (see also \cite[Chapter 10.7]{Tit}). When $a=1/4$ or $1/2$, it holds that
\[
Q(s,1/2) = (2^s + 2^{1-s} -2) \zeta (s), \qquad 
2Q(s,1/4) = (2^{2s} - 2^s + 2^{2-2s} - 2^{1-s}) \zeta (s) 
\]
by (\ref{eq:qq}) (see also \cite[Section 2.2]{NRCQZ}). Therefore, we suppose $0 <a < 1/4$ or $1/4 < a < 1/2$ which implies that $\cos (2\pi a) \ne 0$ (see Lemma \ref{lem:Inlow}). 

\begin{proof}[Proof of Theorem \ref{th:main}]
Let $\nu$ be the number of zeros of $Q(1/2+it,a)$ in the interval $\{0 \le t \le B+k \}$. And let the line $\Re (s) =1/2$ be divided into intervals of length $k$ and for each of the $\nu$ zeros strike out the interval which contains it and the intervals which adjoin this one. Let $S$ be the subset of $\{ A \le t \le B\}$ consisting of points which do not lie in the stricken intervals. Then, the total length of the intervals of $S$ is not less than $B-A-3\nu k$ because a length of at most $3k$ was stricken for each zero. Note that $|I_a (t)| = J_a(t)$ for all $t \in S$. Put $I := \int_S | I_a (t) | dt$. Then, by Lemma \ref{lem:Jest} and the fact that there are no zeros between $t-k$ and $t+k$, we have
\[
I = \int_S J_a (t) \, dt \gg
\int_S (B+k)^{-1/4} e^{-(B+k)\delta/2} \int_{t-k}^{t+k} \bigl| Q(1/2+iu,a) \bigr| du dt .
\]
From Lemmas \ref{lem:Inlow} and \ref{lem:C-S}, we have
\begin{equation*}
\begin{split}
I & \gg
(B+k)^{-1/4} e^{-(B+k)\delta/2} \int_S  \Biggl( 2k \cos |(2\pi a)| - C_1 (a) - \frac{C_2 (a) k^2}{(t-k)^{1/2}} \\ &
\quad - \Biggl| \sum_{a_*=a, 1-a} \sum_{2 \le n < t/a_\star} 
\biggl( \frac{\sin (k \log (n+a_*))}{(n+a_*)^{1/2+it}\log(n+a_*)} + 
\frac{e^{2\pi i a_* n} \sin (k \log n)}{n^{1/2+it}\log n} \biggr) \Biggr| \Biggr) dt \\ & \gg
(B+k)^{-1/4} e^{-(B+k)\delta/2} \Bigl( 2k (B-A-3\nu k) |\cos (2\pi a)| - C_7 (a) B - C_8(a) k^2 B^{1/2} \Bigr) 
\end{split}
\end{equation*}
for some positive constants $C_7 (a)$ and $C_8 (a)$. Let $(B+k)\delta =2$, which can be regarded as a choice of $B$ given $\delta$ and $k$, and let $B-A= \delta^{-1}$ which can be regarded as a choice of $A$. Then, the estimation above becomes
\begin{equation*}
\begin{split}
I  &\gg 
\delta^{1/4} \Bigl( 2k (\delta^{-1}-3\nu k) |\cos (2\pi a)| - 2 C_7 (a) \delta^{-1} - \sqrt{2} C_8(a) k^2 \delta^{-1/2} \Bigr)
\\ &\gg K_1 k \delta^{-3/4} - K_2 k^2 \delta^{1/4} \nu -K_3 \delta^{-3/4} - K_4 k^2 \delta^{-1/4},
\end{split}
\end{equation*}
where $K_1$, $K_2$, $K_3$ and $K_4$ are positive constants (depending on $a$). In contrast, from Lemma \ref{lem:upper}, we have
\begin{equation*}
\begin{split}
I & \le \int_A^B \bigl| I_a (t) \bigr| dt \le 
\biggl( \int_A^B \! 1^2 dt \biggr)^{\!\! 1/2}  \biggl( \int_A^B \bigl| I_a (t) \bigr|^2 dt \biggr)^{\!\! 1/2} \\
& \le B^{1/2} \biggl| \int_{1/2-i\infty}^{1/2+i\infty} \bigl| I_{x,k} (s,a) \bigr|^2 ds \biggr|^{1/2} 
= K_5\frac{(Kk+\varepsilon k^2)^{1/2}}{\delta^{3/4}}. 
\end{split}
\end{equation*}
Therefore, it holds that
\[
K_1 k \delta^{-3/4} - K_2 k^2 \delta^{1/4} \nu - K_3 \delta^{-3/4} - K_4 k^2 \delta^{-1/4}
\le K_5 \delta^{-3/4} (Kk+\varepsilon k^2)^{1/2},
\]
which is equivalent to
\[
\nu \ge \frac{K_1}{K_2} k^{-1} \delta^{-1} - \frac{K_3}{K_2} k^{-2} \delta^{-1} - \frac{K_4}{K_2} \delta^{-1/2}
- \frac{K_5}{K_2} k^{-1} \delta^{-1} \Bigl( \frac{K}{k} + \varepsilon \Bigr)^{\! 1/2} .
\]
We can make the coefficient of $k^{-1} \delta^{-1}$ on the right-hand side positive by choosing $\varepsilon >0$ and $k^{-1}>0$ to be sufficiently small. Hence, with this fixed $k>0$, it has been shown that for all sufficiently small $\delta>0$, the number of roots on the line segment from $1/2$ to $1/2+i2\delta^{-1}$ is greater than $K_6 \delta^{-1} - K_7 \delta^{-1/2}$ with $K_6>0$. 
\end{proof}

\subsection*{Acknowledgments}
The author would like to thank Professors Kohji Matsumoto and Masatoshi Suzuki for their useful advice. 
The author would like to also thank the referee for a careful reading of the manuscript and valuable comments and remarks.
The author was partially supported by JSPS, grant no.~16K05077. 



\begin{thebibliography}{1}
\bibitem{Apo} 
T.~M.~Apostol, \textit{Introduction to Analytic Number Theory}. Undergraduate Texts in Mathematics, Springer, New York, 1976.

\bibitem{BSS}
S.~Baier, K.~Srinivas and U.~K.~Sangale, {\it{A note on the gaps between zeros of Epstein's zeta-functions on the critical line}}. Funct. Approx. Comment. Math. {\bf{57}} (2017),  no.~2, 235--253. 

\bibitem{BH} 
E.~Bombieri and D.~A.~Hejhal, {\it{On the distribution of zeros of linear combinations of Euler products}}. Duke Math.~J. {\bf{80}} (1995), no.~3, 821--862. 


\bibitem{Con}
J.~B.~Conrey, {\it{More than two fifths of the zeros of the Riemann zeta function are on the critical line}}. J.~Reine Angew.~Math. {\bf{399}} (1989), 1--26.

\bibitem{ERZ}
H.~M.~Edwards, \textit{Riemann's zeta function}. Reprint of the 1974 original [Academic Press, New York]. Dover Publications, Inc., Mineola, NY, 2001. 
\bibitem{Ep}
P.~Epstein {\it{Zur Theorie allgemeiner Zetafunktionen.~II}}. Math.~Ann. {\bf{63}} (1906), no.~2, 205--216

\bibitem{Gon}
S.~M.~Gonek, {\it{The zeros of Hurwitz's zeta function on }} $\sigma=1/2$. Analytic number theory (Philadelphia, Pa., 1980), pp. 129--140, Lecture Notes in Math., 899, Springer, Berlin-New York, 1981. 

\bibitem{Ham} 
H.~Hamburger, {\it{\"Uber die Riemannsche Funktionalgleichung der $\zeta$-Funktion}}. (German) Math.~Z. {\bf{10}} (1921), no.~3-4, 240--254. 

\bibitem{HL}
G.~H.~Hardy and J.~E.~Littlewood, {\it{The zeros of Riemann's zeta-function on the critical line}}. Math.~Z. {\bf{10}} (1921),  no.~3-4, 283--317. 

\bibitem{He} 
E.~Hecke, {\it{Herleitung des Euler-Produktes der Zetafunktion und einiger $L$-Reihen aus ihrer Funktionalgleichung.}} Math.~Ann. {\bf{119}} (1944). 266--287. 

\bibitem{KV}
A.~A.~Karatsuba and S.~M.~Voronin, \textit{The Riemann zeta-function}. Translated from the Russian by Neal Koblitz. De Gruyter Expositions in Mathematics, 5. Walter de Gruyter $\&$ Co., Berlin, 1992.

\bibitem{Kn} 
M.~Knopp, {\it{On Dirichlet series satisfying Riemann's functional equation}}. Invent.~Math. {\bf{117}} (1994), no.~3, 361--372. 

\bibitem{KRZ}
P.~K\"uhn, N.~Robles and D.~Zeindler, {\it{On a mollifier of the perturbed Riemann zeta-function}}. J.~Number Theory {\bf{174}}, (2017), 274--321.

\bibitem{JS}
M.~Jutila and K.~Srinivas, {\it{Gaps between the zeros of Epstein's zeta-functions on the critical line}}. Bull.~London Math.~Soc. {\bf{37}} (2005), no.~1, 45--53. 

\bibitem{LauGa}
A.~Laurin\v{c}ikas and R.~Garunk\v{s}tis, {\it{The Lerch zeta-function}}. Kluwer Academic Publishers, Dordrecht, 2002. 

\bibitem{Lev}
N.~Levinson, {\it{More than one third of the zeros of Riemann's zeta function are on $\sigma = 1/2$}}. Adv.~Math.{\bf{13}} (1974), 383--436. 

\bibitem{NRCQZ}
T.~Nakamura, {\it{On the real and complex zeros of the quadrilateral zeta function}}, preprint, arXiv:2001.01981.

\bibitem{NaIn}
T.~Nakamura, {\it{The values of zeta functions composed by the Hurwitz and periodic zeta functions at integers}}, preprint, arXiv:2006.03300.


\bibitem{PT}
H.~S.~Potter and E.~C.~Titchmarsh, {\it{The Zeros of Epstein's Zeta-Functions}}. Proc.~London Math.~Soc. (2) {\bf{39}} (1935), no.~5, 372--384.

\bibitem{Sa}
A.~Sankaranarayanan, {\it{Zeros of quadratic zeta-functions on the critical line}}. Acta Arith. {\bf{69}} (1995), no.~1, 21--38. 

\bibitem{Sie}
C.~Siegel, {\it{Bemerkung zu einem Satz von Hamburger uber die Funktionalgleichung der Riemannschen Zetafunktion}}. Math.~Ann. {\bf{86}} (1922), no.~3-4, 276--279. 

\bibitem{Tit}
E.~C.~Titchmarsh, {\it{The theory of the Riemann zeta-function,}} Second edition. Edited and with a preface by D.~R.~Heath-Brown. The Clarendon Press, Oxford University Press, New York, 1986. 

\end{thebibliography}
\end{document}